\documentclass[11pt,onesided]{article}

\usepackage{amsthm,amsmath,amssymb}
\usepackage{comment}
\usepackage[authoryear,round]{natbib}

\usepackage[dvips,left=1.1in,right=1.1in,top=1in, bottom=1in]{geometry}
\DeclareMathOperator\Var{Var}

\newcommand{\conv}{\operatorname{conv}}

\newcommand{\tot}{\tfrac{1}{2}} % one half in a small-font frac
\newcommand{\scl}[2]{\langle #1,#2 \rangle} % scalar product
\newcommand{\abs}[1]{\left| #1 \right|} % absolute value
\newcommand{\set}[1]{\left\{#1\right\}} % curly brackets
\newcommand{\sset}[1]{\{#1\}} % curly brackets
\newcommand{\sets}[2]{\sset{#1\,:\,#2}} % a set with "such that"
\newcommand{\inds}[1]{ {\mathbf 1}_{\set{#1}}} % indicator of a curly set
\newcommand{\seq}[1]{\sset{#1_n}_{n\in\N}} % a sequence indexed by n

\newcommand{\sq}[2][n]{\{ #2 \}_{#1\in\N}}  % a sequence
\newcommand{\net}[1]{(#1_{\alpha})_{\alpha\in A}} % a net indexed by \alpha
\newcommand{\netb}[1]{\{#1_{\beta}\}_{\beta\in B}} % a net indexed by \beta
\newcommand{\cd}{\text{c\` adl\` ag} } % cadlag phrase
\newcommand{\ft}[2]{#1\dots#2} 
\renewcommand{\ft}[2]{#1,\dots,#2}

\newcommand{\prf}[1]{ ( #1 )_{t\in [0,T]}}

\newcommand{\RN}[2]{\frac{d#1}{d#2}}

\providecommand{\R}{} \renewcommand{\R}{{\mathbb R}}

\newcommand{\N}{{\mathbb N}}
\newcommand{\PP}{{\mathbb P}}

\newcommand{\EE}{{\mathbb E}}
\newcommand{\FF}{{\mathcal F}}

\newcommand{\eps}{\varepsilon}
\newcommand{\ld}{\lambda}

\newcommand{\el}{{\mathbb L}} %l-pees
\newcommand{\lzer}{\el^0}
\newcommand{\lone}{\el^1}

\newcommand{\lpee}{\el^p}

\newcommand{\linf}{\el^{\infty}}
\newcommand{\define}[1]{{\em #1}}

\newcommand\sB{{\mathcal B}}

\newcommand\sC{{\mathcal C}}

\newcommand\sD{{\mathcal D}}

\newcommand\sL{{\mathcal L}}

\newcommand\sO{{\mathcal O}}

\newcommand\sQ{{\mathcal Q}}

\newcommand\sU{{\mathcal U}}

\newcommand{\vp}{\varphi}
\DeclareMathOperator\Fin{Fin}

\DeclareMathOperator\Dom{Dom}

\renewcommand{\net}[1]{\set{#1_{\alpha}}_{\alpha\in A}}

\newcommand{\netd}[1]{\set{#1_{\delta}}_{\delta\in \Fin(A)}}
\newcommand{\netdd}[1]{\set{#1_{D}}_{D\in \Fin(A)}}

\newcommand{\hlzer}{\hat{\el}^0}
\numberwithin{equation}{section}
\theoremstyle{plain}                % title and number in bold, text italic
\newtheorem{theorem}{Theorem}[section]
\newtheorem{lemma}[theorem]{Lemma}
\newtheorem{proposition}[theorem]{Proposition}
\newtheorem{corollary}[theorem]{Corollary}
\theoremstyle{definition}           % title and number in bold, text normal
\newtheorem{definition}[theorem]{Definition}
\newtheorem{example}[theorem]{Example}
\newtheorem{claim}[theorem]{Claim}

\theoremstyle{remark}
\newtheorem{remark}[theorem]{Remark}

\newcommand{\Opt}{{\sO}}
\newcommand{\Xn}{X^{(n)}}

\newcommand{\hXn}{\hat{X}^{(n)}}
\newcommand{\prt}{{\pi}}

\begin{document}

\begin{center}
\Large\sc
Convex compactness and its applications
\end{center}

\bigskip

% indicate corresponding author with \corref{}
\begin{center}
{\bf Gordan {\v Z}itkovi{\' c}}
\\Department of Mathematics
  \\University of Texas at Austin \\1 University Station, Mail Code
  C1200 \\Austin, TX 78712\\
{\tt gordanz@math.utexas.edu\\
 www.ma.utexas.edu/$\sim$gordanz}
\end{center}

\bigskip

\begin{quote}
{\bf Abstract.}
The concept of convex compactness, weaker than the classical notion of
compactness, is introduced and discussed.  It is shown that a large
class of convex subsets of topological vector spaces shares this
property and that is can be used in lieu of compactness in a variety
of cases. Specifically, we establish convex compactness for certain familiar 
classes of subsets of the set of positive random variables under the
topology induced by convergence in probability. 
  
Two applications in infinite-dimensional optimization - attainment of
infima and a version of the Minimax theorem - are given. Moreover, a
new fixed-point theorem of the
Knaster-Kuratowski-Mazurkiewicz-type is derived and used to prove a 
general version of the Walrasian excess-demand theorem. 
\end{quote}

\bigskip

{\footnotesize

\noindent{\em JEL Codes:} C61, C62

\medskip

\noindent{\em 2000 Mathematics Subject Classification:} Primary. 46N10;\
Secondary. 47H10,60G99

\medskip

\noindent{\em Keywords:}
{convex compactness},
{excess-de\-mand theorem},
{Fr\' echet spaces},
{KKM theorem},
{minimax theorem}, 
{optimization},

\medskip

\noindent{\em Acknowledgements:} This work was supported in part by
the National Science Foundation under award number DMS-0706947.  Any
opinions, findings and conclusions or recommendations expressed in
this material are those of the author and do not necessarily reflect
those of the National Science Foundation.  The author would like to
thank Nigel Kalton for enlightening comments on the history of the
subject.  Many thanks to Andreas Hamel, Constantinos Kardaras,
Mi\-kl\' os R\' a\-so\-nyi, Walter
Schachermayer, Ted Odell and Mark Owen and  for
fruitful discussions and good advice.  } 
\pagebreak

\section{Introduction}
The raison d'\^etre of the present paper is to provide the community
with a new formulation and a convenient toolbox for a number of
interrelated ideas centered around compactness substitutes in
infinite-dimensional spaces that occur frequently in mathematical
economics and finance. Our goal is to keep the presentation short and
to-the-point and derive a number of (still quite general)
mathematical theorems which we believe will find their use in the
analysis of various applied problems in finance and economics. For
this reason, we do not include worked-out examples in specific
economic models, but paint a broader picture and point to some common
areas of interest (such as infinite-dimensional optimization or
general-equilibrium theory) where we believe our results can be
successfully applied.

\bigskip

Compactness in infinite-dimensional topological vector
spaces comes at a cost: either the size of the set, or the strength of
the topology used have to be severely restricted.  Fortunately,
the full compactness requirement is often not necessary for applications,
especially in mathematical economics where a substantial portion of
the studied objects exhibits at least some degree of convexity. This
fact has been heavily exploited in the classical literature, and
refined recently in the work of \cite{DelSch94, DelSch99, KraSch99}
and others in the field of mathematical finance (e.g., in the context
of the fundamental theorem of asset pricing and utility maximization
theory in incomplete markets). More specifically, the above authors,
as well as many others, used various incarnations of the beautiful
theorem of \cite{Kom67} (see, also, \cite{Sch86}) to extract a
convergent sequence of convex combinations from an arbitrary (mildly
bounded) sequence of non-negative random variables (see, e.g.,
\cite{Nik71,BuhLoz73} for some of the earliest references).  In the
context of convex optimization problems this procedure is often as
versatile as the extraction of a true subsequence, but can be
applied in a much larger number of situations.

\medskip

The initial goal of this note is to abstract the precise property -
called {\em convex compactness} - of the space $\lzer_+$ (with
$\lzer_+$ denoting the set of all a.s.-equivalence classes of
non-negative random variables, equipped with the complete metric
topology of convergence in probability) which allows for the rich
theory mentioned above. Our first result states that, in the class of
convex subsets of $\lzer_+$, the convexly compact sets are precisely
the bounded and closed ones.  An independently-useful abstract
characterization in terms or generalized sequences is given. We stress
that the language of probability theory is used for convenience only,
and that all of our results extend readily to abstract measure spaces
with finite measures.

Recent research in the field of mathematical finance has singled out
$\lzer_+$ as especially important - indispensable, in fact. Unlike in
many other applications of optimal (stochastic) control and calculus
of variations, the pertinent objective functions in economics and
finance - the utility functions - are too ``badly-behaved'' to allow
for the well-developed and classical $\lpee$ theory (with $p\in
[1,\infty]$) to be applied. Moreover, the budget-type constraints
present in incomplete markets are typically posed with respect to a
class of probability measures (which are often the martingale measures
for a particular financial model) and not with respect to a single
measure, as often required by the $\lpee$ theory. The immediate
consequence of this fact is that the only natural functional-analytic
framework must be, in a sense, measure-free. In the hierarchy of the
$\lpee$ spaces, this leaves the two extremes: $\lzer$ and $\linf$. The
smaller of the two - $\linf$ - is clearly too small to contain even
the most important special cases (the solution to the Merton's
problem, for example), so we are forced to work with $\lzer$. The
price that needs to be paid is the renouncement of a large number of
classical functional-analytic tools which were developed with
locally-convex spaces (and, in particular, Banach spaces) in
mind. $\lzer$ fails the local-convexity property in a dramatic
fashion: if $(\Omega,\FF,\PP)$ is non-atomic, the topological dual of
$\lzer$ is trivial, i.e., equals $\set{0}$. Therefore, a new set of
functional-analytic tools which do not rely on local convexity are
needed to treat even the most basic problems posed by finance and
economics: pricing, optimal investment, and, ultimately, existence of
equilibria.

In addition to existence results for the usual optimization problems
in quantitative finance (see \cite{Mer69,Mer71}, \cite{Pli86},
\cite{CoxHua89}, \cite{KarLehShrXu90}, \cite{HePea91a},
\cite{KraSch99}, \cite{CviSchWan01}, \cite{KarZit03}), some recent
results dealing with stability properties of those problems testify
even stronger to the indispensability of $\lzer_+$ (and the related
convergence in probability) in this context. As \cite{LarZit06a} and
\cite{KarZit07} have shown, the output (i.e., the demand function) of
such optimization procedures is a continuous transformation of the
primitives (a parametrization of a set of models) when both are viewed
as subsets of $\lzer_+$. Moreover, there are simple examples showing
that such continuity fails if the space $\lzer_+$ is upgraded to an
$\lpee$-space, for $p\geq 1$. In addition to their intrinsic interest,
such stability is an important ingredient in the study of equilibria
in incomplete financial (stochastic) models. Indeed, the continuity in
probability of the excess-demand function enters as condition {\em
  1.}~in Theorem \ref{thm:excess-dem}. Let us also mention that the
most relevant classical approaches to the general-equilibrium theory
(see \cite{Bew72} or \cite{MasZam91}) depend on the utility functions
exhibiting a certain degree of continuity in the Mackey topology. This
condition is not satisfied in general in the Alt-von
Neumann-Morgenstern setting for utility functions satisfying the Inada
conditions (the HARA and CRRA families, e.g.). They do give rise to
$\lzer_+$-continuous demand functions under mild regularity
conditions.

\begin{comment}
Other
examples where convex compactness may hold without
classical compactness are presented. The 
 most important one deals with the positive
orthant of the space of \cd quasimartingales on a filtered probability
space.  Using the metric of convergence in
(a suitably-defined) measure on the product space $\Omega\times [0,T]$
and the order relation in which the
positive cone consists of all non-negative \cd supermartingales,   
we provide a
necessary condition for convex compactness.
\end{comment}

\medskip

Another contribution of the present paper is the realization that,
apart from its applicability to $\lzer_+$, convex compactness is a
sufficient condition for several well-known and widely-applicable
optimization results to hold under very general conditions. We start
with a simple observation that lower semicontinuous, quasiconvex and
appropriately coercive functionals attain their infima on convexly
compact sets. In the same spirit, we establish a version of the
Minimax Theorem where compactness is replaced by convex compactness.

The central application of convex compactness is a generalization of
the classical ``fixed-point'' theorem of Knaster, Kuratowski and
Mazurkiewicz (see \cite{KnaKurMaz29}) to convexly compact sets.  In
addition to being interesting in its own right, it is used to give a
variant of the Walrasian excess-demand theorem of general-equilibrium
theory based on convex compactness.  This result can serve as a
theoretical basis for various equilibrium existence theorems where
regularity of the primitives implies only weak regularity (continuity
in probability) of the excess-demand functions. As already mentioned
above, that is, for instance, the case when the agents' preferences
are induced by Alt-von Neumann-Morgenstern expected utilities which
satisfy Inada conditions, but are not necessarily Mackey-continuous.

The paper is structured as follows: after this introduction, section
2.~defines the notion of convex compactness and provides necessary
tools for its characterization. Section 3.~deals with the main source
of examples of convex compactness - the space $\lzer_+$ of all
non-negative random variables. The applications are dealt with in section 4.

\section{Convex compactness}

\subsection{The notion of convex compactness}
Let $A$ be a non-empty set. The set $\Fin(A)$ consisting of all
non-empty finite subsets of $A$ carries a natural structure of a
partially ordered set when ordered by inclusion. Moreover, it is
 a directed set, since $D_1, D_2 \subseteq D_1\cup D_2$ for
any $D_1,D_2\in \Fin(A)$.
\begin{definition}
\label{def:convex-compact}
A convex subset $C$ of a topological vector space $X$  is said to be
{\em convexly compact} if for any non-empty set $A$ and any family
$\net{F}$ of { closed and convex} subsets of $C$, the condition
\begin{equation}%
  \label{equ:FI}
  % \nonumber
  \begin{split}
\forall\, D\in\Fin(A),\, \bigcap_{\alpha \in D} F_{\alpha} \not =
\emptyset  
    \end{split}
\end{equation} 
implies 
\begin{equation}%
\label{equ:I}
%    \nonumber 
    \begin{split}
\bigcap_{\alpha\in A} F_{\alpha}\not = \emptyset.
    \end{split}
\end{equation}
\end{definition}

 Without the additional restriction that the sets $\net{F}$ be
    convex, Definition \ref{def:convex-compact} - 
    postulating the finite-intersection property for families of
    closed and convex sets - would be equivalent to
    the classical definition of compactness. It is, therefore,
    immediately clear that any convex and compact subset of a
    topological vector space is convexly compact. 
\begin{example}[{\bf Convex compactness without compactness}] 
 Let $L$ be a locally-convex topological
    vector space, and let $L^*$ be the topological dual of $L$,
    endowed with some compatible topology $\tau$, possibly different
    from the weak-* topology $\sigma(L^*,L)$.  For a neighborhood $N$
    of $0$ in $L$, define the set $C$ in the topological dual $L^*$ of
    $L$ by
    \[ C=\sets{x^*\in L}{ \scl{x}{x^*}\leq 1,\ \forall\, x\in N}.\] In
    other words, $C=N^{\circ}$ is the polar of $N$. By the
    Banach-Alaoglu Theorem, $C$ is compact with respect to the weak-*
    topology $\sigma(L^*,L)$, but it may not be compact with respect
    to $\tau$. On the other hand, let $\net{F}$ be a non-empty family
    of convex and $\tau$-closed subsets of $C$ with the
    finite-intersection property (\ref{equ:FI}).  It is a
    classical consequence of the Hahn-Banach Theorem that the
    collection of closed and convex sets is the same for all
    topologies consistent with a given dual pair. Therefore, the sets
    $\net{F}$ are $\sigma(L^*,L)$-closed, and the
    relation (\ref{equ:I}) holds by the aforementioned
    $\sigma(L^*,L)$-compactness of $C$.
\end{example}
\subsection{A characterization in terms of generalized sequences} 
The classical theorem of \cite{Kom67} states that any norm-bounded
sequence in $\lone$ admits a subsequence whose Ces\` aro sums converge
a.s. The following characterization draws a parallel between Koml\'os'
theorem and the notion of convex compactness. It should be kept in
mind that we forgo equal (Ces\' aro) weights guaranteed by Koml\' os'
theorem, and settle for generic convex combinations.  We remind the
reader that for a subset $C$ of a vector space $X$, $\conv C$ denotes
the smallest convex subset of $X$ containing $C$.
\begin{definition}
\label{def:net-conv}
  Let $\net{x}$ be a net in a vector space $X$. A net $\netb{y}$ is
  said to be a \define{subnet of convex combinations of} $\net{x}$ if
  there exists a mapping $D:B\to \Fin(A)$ such that 
  \begin{enumerate}
  \item $y_{\beta}\in \conv\sets{x_{\alpha}}{\alpha\in D(\beta)}$ for
    each $\beta\in B$, and
  \item for each $\alpha\in A$ there exists $\beta\in B$ such that
    $\alpha'\succeq \alpha$ for each $\alpha'\in\bigcup_{\beta'\succeq
      \beta} D(\beta')$.
  \end{enumerate}
\end{definition}
\begin{proposition} 
\label{pro:cc-char}
A closed and convex subset $C$ of a topological
  vector space 
$X$ is convexly compact if and only if for any net
  $\net{x}$ in $C$ there exists a subnet $\netb{y}$ of convex
  combinations of $\net{x}$ such that $y_{\beta}\to y$ for some $y\in
  C$.
\end{proposition}
\begin{proof}
  {\bf $\Rightarrow$} Suppose, first, that $C$ is convexly compact, and
  let $\net{x}$ be a net in $C$.  For $\alpha\in A$ define the closed
  and convex set $F_{\alpha}\subseteq C$ by
  \[ F_{\alpha}=\overline{\conv}\sets{x_{\alpha'}}{\alpha'\succeq
    \alpha},\] where, for any $G\subseteq X$, 
 $\overline{\conv}\, G$ denotes the closure of $\conv G$.
  By convex compactness of $C$, there exists $y\in
  \cap_{\alpha\in A} F_{\alpha}$.  Define $B=\sU\times A$, where $\sU$
  is the collection of all neighborhoods of $y$ in $X$. The binary
  relation $\preceq_B$ on $B$, defined by \[
  \text{$(U_1,\alpha_1)\preceq_B (U_2, \alpha_2)$ if and only if $U_2
    \subseteq U_1$ and $\alpha_1\preceq \alpha_2$},\] is  a
  partial order under which $B$ becomes a directed set.

  By the construction of the family $\net{F}$, for each
  $\beta=(U,\alpha)\in B$ there exists a finite set $D(\beta)$ such
  that $\alpha'\succeq \alpha$ for each $\alpha'\in D(\beta)$ and an
  element $y_{\beta}\in \conv\sets{x_{\alpha'}}{\alpha'\in D(\beta)}
  \cap U$. It is evident now that $\netb{y}$ is a subnet of convex
  combinations of $\net{x}$ which converges towards $y$.

{\bf $\Leftarrow$ } Let $\net{F}$ be a family of closed and convex
subsets of $C$ satisfying (\ref{equ:FI}). For $\delta\in \Fin(A)$, we
set
\[ G_{\delta}=\bigcap_{\alpha\in\delta} F_{\alpha}\not=\emptyset,\] so
that $\netd{G}$ becomes an non-increasing net in
$(2^C,\subseteq)$, in the sense that $G_{\delta_1} \subseteq
G_{\delta_2}$ when $\delta_1 \supseteq \delta_2$. For each $\delta\in
\Fin(A)$,  we pick $x_{\delta}\in G_{\delta}$. By the assumption, there
exists a subnet $\netb{y}$ of convex combinations of $\netd{x}$
converging to some $y\in C$. More precisely, in the present setting,
conditions
(1) and (2) from
Definition \ref{def:net-conv} yield the
existence of a directed set $B$ and a function $D: B\to \Fin(\Fin(A))$
such that:
\begin{itemize}
\item[(1)] $y_{\beta}\in \conv\sets{x_{\delta}}{\delta\in
    D(\beta)}\subseteq \conv \cup_{\delta\in D(\beta)} G_{\delta}$;  in
  particular, $y_{\beta}\in G_{\delta}$ for each $\delta$ such that
  $\delta \subseteq \cap_{\delta'\in D(\beta)} \delta'$. 
\item[(2)] For each $\delta\in \Fin(A)$ there exists $\beta\in B$ such that 
$\delta \subseteq \delta'$ for all $\delta'\in \cup_{\beta'\succeq
  \beta} D(\beta')$.  In particular, if we set $\delta=\set{\alpha}$ for
an arbitrary $\alpha\in A$, we can assert the existence of
$\beta_{\alpha}\in B$ such that $\alpha\in \delta'$ for any
$\delta'\in \cup_{\beta'\succeq \beta_{\alpha}} D(\beta')$. 
\end{itemize}
Combining (1) and (2) above, we get that for each $\alpha\in A$ there
exists $\beta_{\alpha} \in B$ such that $y_{\beta}\in
G_{\set{\alpha}}=F_{\alpha}$ for all
$\beta\succeq\beta_{\alpha}$. Since $y_{\beta}\to y$ and $F_{\alpha}$
is a closed set, we necessarily have $y\in F_{\alpha}$  and, thus,
$y\in \cap_{\alpha\in A} F_{\alpha}$. Therefore,  $\cap_{\alpha\in A}
F_{\alpha}\not=\emptyset$. 
\end{proof}

\section{A space of random variables}

\subsection{Convex compactness in $\lzer_+$}
Let $(\Omega,\FF,\PP$) be a complete probability space.   The positive
orthant $\sets{f\in\lzer}{f\geq 0,\text{ a.s.}}$ will be denoted by
$\lzer_+$.
A set
$A\subseteq\lzer$ is said to be \define{bounded in $\lzer$} (or
\define{bounded in probability}) if
\[ \lim_{M\to\infty} \sup_{f\in A} \PP[\abs{f}\geq M]=0.\]
Unless specified
otherwise, any mention of convergence on $\lzer$ will be under the
topology of convergence in probability, induced by the
translation-invariant metric
$ d(f,g)=\EE[ 1\wedge \abs{f-g}]$, 
making $\lzer$ a Fr\' echet space (a topological vector
space admitting a complete compatible metric). 
It is well-known, however,  
that $\lzer$ is generally {\em not} a locally-convex
topological vector space. In fact, when $\PP$ is non-atomic, 
$\lzer$ admits no non-trivial continuous
linear functionals (see \cite{KalPecRob84}, Theorem 2.2, p.~18).  

In addition to not being locally-convex, the space $\lzer$ has poor
compactness properties. Indeed, let $C$ be a closed
uniformly-integrable ($\sigma(\lone,\linf)$-compact) set in
$\lone\subseteq \lzer$.  By a generalization of the dominated
convergence theorem, the topology of convergence in probability and
the strong $\lone$-topology on $C$ coincide; a statement of this form
is sometimes know as Osgood's theorem (see, e.g., Theorem 16.14,
p.~127, in \cite{Bil95}).  Therefore, the question
of $\lzer$-compactness reduces to the question of strong
$\lone$-compactness for such sets. It is a well-known fact that $C$ is
$\lone$-compact (equivalently, $\lzer$-compact) if and only if it
satisfies a non-trivial condition of \cite{Gir91} called the {\em
  Bocce criterion}.  The following result shows that, however, a much
larger class of sets shares the convex-compactness property. In fact,
there is a direct analogy between the present situation and the
well-known characterization of compactness in Euclidean spaces.
\begin{theorem}
\label{thm:main}
A closed and  convex subset $C$ of \,$\lzer_+$ is
convexly compact if and only if it is bounded in probability. 
\end{theorem}
\begin{proof}
{\bf $\Leftarrow$ }
  Let $C$ be a convex, closed and bounded-in-probability subset of
  $\lzer_+$,
  and let $\net{F}$ be a family of closed and convex subsets of
  $C$ satisfying \eqref{equ:FI}. 
  For $D\in \Fin(A)$ we define
  \[ G_D=C, \text{ when } D=\emptyset,\text{ and } 
G_D=\cap_{\alpha\in D} F_{\alpha}, \text{ otherwise,}
\] and fix an
  arbitrary $f_D\in G_D$.
  With $\vp(x)=1-\exp(-x)$, we set
  \[ u_D=\sup\sets{ \EE[\vp(g)] }{ g\in
    \conv\sets{f_{D'}}{D'\supseteq D} },\] so that $0\leq u_D\leq 1$
  and $u_{D_1} \geq u_{D_2}$, for $D_1\subseteq D_2$.
Seen as a net on the directed set $(\Fin(A),\subseteq)$, $\netdd{u}$ is
monotone and bounded, and therefore convergent, i.e.,
  $u_D\to u_{\infty}$, for some $u_{\infty}\in [0,1]$. 
 Moreover, for each
  $D\in \Fin(A)$ we can choose $g_D\in \conv\sets{g_{D'}}{D\subseteq D'}$ so
  that
  \[ u_D\geq \gamma_D\triangleq \EE[\vp(g_D)]\geq
  u_D-\frac{1}{\# D},\] where $\#D$ denotes the number of elements
  in $D$.  Clearly,  $\gamma_D\to u_{\infty}$.

  The reader is invited to check that simple analytic properties of
  the function $\vp$ are enough to prove the following statement: for
  each $M>0$ there exists $\eps=\eps(M)>0$, such that
  \begin{multline}
    \nonumber \text{if }\ x_1,x_2\geq 0,\  
\abs{x_1-x_2}\geq \tfrac{1}{M}\, \text{ and }\, 0\leq
    \min(x_1,x_2) \leq M, \\ \text{ then } \vp(\tot (x_1+x_2)) \geq
    \tot (\vp(x_1)+\vp(x_2))+\eps.
  \end{multline}
It follows that for any $D_1,D_2\in \Fin(A)$ we have
  \begin{multline}
    \nonumber \eps \PP\Big[ \abs{g_{D_1}-g_{D_2}} \geq \tfrac{1}{M},\
    \min(g_{D_1},g_{D_2})\leq M\Big] \leq \\  \leq \EE\left[ \vp\Big(\tot
    (g_{D_1}+g_{D_2}) \Big)\right]- \tot\left( \EE\big[\vp(g_{D_1})\big]+
\EE\big[\vp(g_{D_2})\big]\right).
  \end{multline}
  The random variable $\tot(g_{D_1}+g_{D_2})$ belongs to
  $\conv\sets{f_{D'}}{D'\supseteq D_1\cap D_2}$, so  
\[ \EE[ \vp(\tot (g_{D_1}+g_{D_2}) )]\leq u_{D_1\cap D_2}.\] 
  Consequently,
  \begin{equation}%\label{}
    \nonumber 
    \begin{split}
      0 \leq \eps \PP[ \abs{g_{D_1}-g_{D_2}} \geq 1/M,\
      \min(g_{D_1},g_{D_2})\leq M] \leq \eta_{D_1,D_2},
    \end{split}
  \end{equation}
  where $\eta_{D_1,D_2} = u_{D_1\cap D_2} -
  \tot(u_{D_1}+u_{D_2})+\frac12
  (\frac{1}{\#D_1}+\frac{1}{\#D_2})$.
  Thanks to the boundedness in probability of the set $C$, for $\kappa>0$,
  we can find $M=M(\kappa)>0$ such that $M>1/\kappa$ and $\PP[f\geq
  M]<\kappa/2$ for any $f\in C$.
  Furthermore, let $D(\kappa)\in \Fin(A)$ be such that
  $u_{\infty}+\eps(M) \kappa/4\geq u_D \geq u_{\infty}$
  for all $D\supseteq D(\kappa)$, and $\# D(\kappa)>4/(\eps(M)
  \kappa)$.
  Then, for $D_1, D_2 \supseteq D(\kappa)$ we have
  \begin{equation}%\label{}
    \nonumber 
    \begin{split}
      \PP[ \abs{g_{D_1}-g_{D_2}} \geq \kappa] & \leq \PP[
      \abs{g_{D_1}-g_{D_2}} \geq \tfrac{1}{M}, \min(g_{D_1},g_{D_2})
      \leq M ]+
      \PP[ \min(g_{D_1},g_{D_2})\geq M] \\
      & \leq \tfrac{1}{\eps(M)} \Big( u_{D_1\cap D_2} -
      \tot\big(u_{D_1}+u_{D_2}\big)+\tot (\tfrac{1}{\#
        D_1}+\tfrac{1}{\#D_2})\Big) \leq \kappa.
    \end{split}
  \end{equation}
  In other words, $\{ g_{D} \}_{D\in \Fin(A)}$ is a Cauchy net in $\lzer_+$
  which, by completeness, admits a limit $g_{\infty}\in \lzer_+$.  By
  construction and convexity of the sets $F_{\alpha}$, $\alpha\in A$,
  we have $g_{D}\in F_{\alpha}$ whenever $D\supseteq \set{\alpha}$. By
  closedness of $F_{\alpha}$, we conclude that $g_{\infty} \in
  F_{\alpha}$, and so, $g_{\infty} \in \cap_{\alpha\in A} F_{\alpha}$.

\medskip

{\bf $\Rightarrow$ }
It remains to show that convexly compact sets in $\lzer_+$ are
necessarily bounded in probability. Suppose, to the contrary, that
$C\subseteq \lzer_+$
is convexly compact, but not bounded in probability. Then, 
 there
exists a constant $\eps\in (0,1)$ and a sequence $\seq{f}$ in $C$ such that 
\begin{equation}%
\label{equ:prob-bound}
%    \nonumber 
    \begin{split}
\PP[f_n\geq n]>\eps,\text{ for all } n\in\N. 
    \end{split}
\end{equation}
By Proposition
\ref{pro:cc-char}, there exists a subnet $\netb{g}$ of convex
combinations of $\seq{f}$ which converges to some $g\in C$. In
particular, for each $n\in\N$ there exists $\beta_n\in B$ such that 
$g_{\beta}$ can be written as a finite convex combination of the
elements of the set $\sets{f_{m}}{m\geq n}$, for any $\beta\succeq
\beta_n$.  Using (\ref{equ:prob-bound}) and  
Lemma 9.8.6., p.~205 in \cite{DelSch06}, we get the following estimate
\begin{equation}%
\label{equ:estimate-g}
%    \nonumber 
    \begin{split}
\PP[ g_{\beta}\geq \tfrac{n\eps}{2}] \geq \tfrac{\eps}{2}, \text{ for
  all }\beta\succeq \beta_n. 
    \end{split}
\end{equation}
Therefore, 
\begin{equation}%\label{}
    \nonumber 
    \begin{split}
\PP[ g\geq \tfrac{n\eps}{4}] & \geq 
\PP[g_{\beta}\geq
\tfrac{n\eps}{2}]-
\PP[\abs{g-g_{\beta}}>\tfrac{n\eps}{4}]
\geq \tfrac{\eps}{2}-\PP[\abs{g-g_{\beta}}>\tfrac{\eps}{4}]>\tfrac{\eps}{4},
    \end{split}
\end{equation}
for all ``large enough'' $\beta\in B$. Hence, $\PP[g=+\infty]>0$ - a
contradiction with the assumption $g\in C$. 
\end{proof}
\section{Applications}
Substitution of the strong notion of compactness for a weaker notion
of convex compactness opens a possibility for extensions of several
classical theorems to a more general setting. In the sequel, let $X$
denote a generic topological vector space.
\subsection{Attainment of infima for convexly coercive
  functions}
We set off  with a simple claim that convex and appropriately regular
functionals attain their infima  on convex-compact sets. This fact (in a
slightly different form) has been observed and used, e.g.,  in \cite{KraSch99}.

  For a function $G:X \to (-\infty,\infty]$ and $\ld\in
  (-\infty,\infty]$,  
we define the \define{$\ld$-lower-contour set}
  $L_G(\ld)$ as $L_G(\ld)=\sets{x\in X}{G(x)\leq\ld}$. The
  \define{effective domain} $\Dom(G)$ of $G$ is defined as
$\Dom(G)=\cup_{\ld<\infty} L_G(\ld)$. 

\begin{definition} 
\label{def:conv-coer}
  A function $G:X\to (-\infty,\infty]$ is said to be 
\define{convexly coercive} if $L_G(\ld)$ is
  convex and closed for all $\ld\in (-\infty,\infty]$, and there exists
  $\ld_0\in (-\infty,\infty]$ 
such that $L_G(\ld_0)$ is non-empty and convexly compact.
\end{definition}
\begin{remark}
  By Theorem \ref{thm:main}, in the special case when $X\subseteq \lzer$, and
  $\Dom(G)$ $\subseteq \lzer_+$, 
  convex coercivity is implied by the conjunction of the 
following three conditions:
  \begin{enumerate}
  \item \define {weak coercivity:} there exists 
$\ld_0\in (-\infty,\infty]$ such that
    the $L_G(\ld_0)$ is bounded-in-probability,
  \item \define{lower semi-continuity:} $L_G(\ld)$ is closed for each
     $\ld\in (-\infty,\infty)$, and
  \item \define{quasi-convexity:} $L_G(\ld)$ is convex for each
$\ld_0\in (-\infty,\infty)$ 
(a condition automatically satisfied when $G$ is convex). 
  \end{enumerate}
\end{remark}

\begin{lemma} Each bounded-from-below 
convexly coercive function $G:X \to (-\infty,\infty]$ attains its 
 infimum on $X$.
\end{lemma}
\begin{proof}
Let $r_0=\inf\sets{G(x)}{x\in X}$.   
  Let $a_0=G(x_0)$, for some  $x_0\in L_G(\ld_0)$, with $\ld_0$
 as in Definition \ref{def:conv-coer} above.
 If $a_0=r_0$, we are done.  Suppose, therefore, that
  $a_0>r_0$. For $a\in (r_0,a_0]$ we define
  \[ F_a=\sets{ x\in X}{G(x)\leq a_0}\not=\emptyset.\]
The family $\sets{F_a}{a\in (r_0,a_0]}$ of convex and closed subsets
of the convex-compact set 
$L_G(\ld_0)$  is  nested; in particular, it
has the finite intersection property (\ref{equ:FI}). Therefore, 
there exists $x^*_0\in \cap_{a\in
  (r_0,a_0]} F_{a}$. It follows that
 $G(x^*_0)=r_0$, and so the minimum of $G$ is attained at $x^*_0$. 
\end{proof}

\begin{remark}
  Let $\Phi:[0,\infty)\to [0,\infty)$ be a convex and lower
  semi-continuous function. Define the mapping $G:\lzer_+\to
  [0,\infty]$ by
  \[ G(f)= \EE[ \Phi(f) ], \ f\in\lzer_+.\] $G$ is clearly convex and
  Fatou's lemma implies that it is lower semi-continuous.  In
  order to guarantee its convex coercivity, one can either restrict the
  function $G$ onto a convex, closed and bounded-in-probability subset
  $B$ of $\lzer$ (by setting $G(f)=+\infty$ for $f\in B^c$), or impose
  growth conditions on the function $\Phi$. A simple example of
  such a condition is the following (which is, in fact, equivalent to
  $\lim_{x\to\infty} \Phi(x)=+\infty$):
\begin{equation}%
  \label{equ:cond-one}
  % \nonumber
    \begin{split}
      \liminf_{x\to\infty} \frac{\Phi(x)}{x}>0.
    \end{split}
  \end{equation}
  Indeed, if \eqref{equ:cond-one} holds, then there exist constants
  $D\in \R$ and $\delta>0$ such that $\Phi(x)\geq D+\delta
  x$. Therefore, for $f_0\equiv 1$, with $c= G(f_0)=\Phi(1)$ we have
  \[ \sets{ f\in\lzer_+}{G(f)\leq c} \subseteq \sets{ f\in\lzer_+}{
    \EE[f]\leq (c-D)/\delta}.\] This set on the right-hand side above
  is bounded in $\lone$
  and, therefore, in probability.
\end{remark}

\subsection{A version of the theorem of Knaster, Kuratowski and
  Mazurkiewicz}
The celebrated theorem of Knaster, Kuratowski and Mazurkiewicz 
originally stated for  finite-dimensional
simplices, is commonly considered as a mathematical basis for the
general equilibrium theory of mathematical economics. 
\begin{theorem}[Knaster, Kuratowski and Mazurkiewicz (1929)]
  \label{thm:KKM-original} 
  \noindent Let $S$ be the unit simplex in $\R^{d}$, $d\geq 1$, with
  vertices ${x_1,x_2,\dots, x_d}$.   Let $\{F_i\,:\, i=\ft{1}{d}\}$, be a
  collection of closed subsets of $S$ such that 
\[ \conv\set{x_{i_1},\dots, x_{i_k}} \subseteq
    \cup_{j=1}^k F_{i_j}\text{ for any subset }\set{i_1,\dots,
      i_k}\subseteq \set{\ft{1}{d}}.\]
Then, \[\bigcap_{i=1}^d F_{i}\not=\emptyset.\] 
\end{theorem}
Before stating a
 useful, yet simple, corollary to Theorem \ref{thm:KKM-original}, we
 introduce 
the KKM-property:
\begin{definition}
  \label{def:KKM-property}
  Let $X$ be a vector space and let $B$ be its non-empty subset.  A
  family $\set{F(x)}_{x\in B}$ of subsets of $X$
  is said to have the
  {\bf Knaster-Kuratowski-Mazurkiewicz (KKM) property} if
  \begin{enumerate}
  \item \label{ite:KKM-one} $F(x)$ is closed for each $x\in
    B$, and
  \item \label{ite:KKM-two} $\conv\set{x_1,\dots, x_n} \subseteq
    \cup_{i=1}^n F(x_i)$, for any finite set $\set{x_1,\dots,
      x_n}\in B$.
  \end{enumerate}
\end{definition}
\begin{corollary}
\label{cor:KKM}
 Let $X$ be a vector space, and let $B$ be a non-empty subset of $X$.
 Suppose that a 
  family $\set{F(x)}_{x\in B}$ of subsets of $X$ has
  the KKM property. Then,
\[ \bigcap_{i=1}^n F(x_i)\not=\emptyset,\]
for any finite set $\set{x_1,\dots, x_n}\subseteq B$.
\end{corollary}
The literature on fixed points abounds with extensions of Theorem
\ref{thm:KKM-original} to various locally-convex settings
(see, e.g.,  Chapter I, \S 4 in \cite{GraDug03}). To the best
of our knowledge, this extension has not been made to the class of
Fr\' echet spaces (and $\lzer$ in particular). The following
result is a direct consequence of the combination of Theorem
\ref{thm:KKM-original} and Theorem \ref{thm:main}. The reader can
observe that the usual assumption of compactness has been replaced by
convexity and convex compactness.
\begin{theorem}
  \label{thm:KKM}
 Let $X$ be a topological 
vector space, and let $B$ be a non-empty subset of $X$.
 Suppose that a 
  family $\set{F(x)}_{x\in B}$ of {convex} 
subsets of $X$, indexed by $B$, has
  the KKM property. If there exists $x_0\in B$ such that $F(x_0)$ is
  convexly compact, then 
  \[ \bigcap_{x\in B} F(x) \not= \emptyset.\] 
\end{theorem}
\begin{proof}
Define $\tilde{F}(x)=F(x)\cap F(x_0)$, for $x\in B$. By Corollary
\ref{cor:KKM}, any finite subfamily of 
$\{\tilde{F}(x)\}_{x\in B}$ has a non-empty intersection. Since, all $\tilde{F}(x)$, $x\in B$ are convex and closed 
subsets of the convexly compact set
$F(x_0)$, we have 
\[\cap_{x\in
  B} F(x)=\cap_{x\in B} \tilde{F}(x)\not=\emptyset.\] 
\end{proof}
\subsection{A minimax theorem for $\lzer_+$} 
Since convexity already appears naturally in the classical Minimax
theorem, one can replace the usual compactness by convex compactness
at little cost. 
\begin{theorem}
  Let $C,D$ be two convexly compact (closed, convex and
  bounded-in-probability) subsets of $\lzer_+$, and let $\Phi:C\times
  D\to \R$ be a function with the following properties:
  \begin{enumerate}
  \item $f\mapsto \Phi(f,g)$ is concave and upper semi-continuous for all
    $g\in D$, 
  \item $g\mapsto \Phi(f,g)$ is convex  and lower  semi-continuous for all
    $f\in C$.
  \end{enumerate}
Then,  there exists a pair $(f_0,g_0)\in C\times D$ such that $(f_0,g_0)$ is a
saddle-point of $\Phi$, i.e.,  
\[ \Phi(f,g_0)\leq \Phi(f_0,g_0) \leq \Phi(f_0,g),\text{ for all } (f,g)\in
C\times D.\]
Moreover, 
\begin{equation}%
\label{equ:supinf}
%    \nonumber 
    \begin{split}
 \sup_{f\in C} \inf_{g\in D} \Phi(f,g)=\inf_{g\in D} \sup_{f\in C}
 \Phi(f,g).
    \end{split}
\end{equation} 
\end{theorem}
\begin{proof}
  Let $(\hat{\Omega},\hat{\FF},\hat{\PP})$ be a direct sum of
  $(\Omega,\FF,\PP)$ and a copy of itself. More precisely, we set
  $\hat{\Omega}=\set{1,2}\times\Omega$; $\hat{\FF}$ is the
  $\sigma$-algebra on $\hat{\Omega}$ generated by the sets of the form
  $\set{i}\times A$, $i=1,2$, $A\in\FF$, and $\hat{\PP}$ the unique
  probability measure on $\hat{\FF}$ satisfying
  $\hat{\PP}[\set{i}\times A] =\tot \PP[A]$, for $i=1,2$ and 
  $A\in\FF$. Then, a pair $(f,g)$ in $\lzer_+(\Omega,\FF,\PP)\times
  \lzer_+(\Omega,\FF,\PP)$ can be identified with the element $f\oplus g$
  of $\hlzer_+=\lzer_+(\hat{\Omega},\hat{\FF},\hat{\PP})$ in the
  following way: $(f\oplus g) (i,\omega) =f(\omega)$ if $i=1$, and
  $(f\oplus g)(i,\omega)=g(\omega)$ if $i=2$. 

  For $f\oplus g\in \hlzer_+$, define $C\oplus D=\sets{f\oplus g\in
    \hlzer_+}{(f,g)\in C\times D}$, together with the family of its subsets
\[ G_{f\oplus g}=\sets{f'\oplus g'\in C\oplus D}{ 
\Phi(f,g')-\Phi(f',g)\leq 0},\ f\oplus g\in C\oplus D.\] 
By properties {\em 1.} and {\em 2.} of the
function $\Phi$ in the statement of the theorem, the set $G_{f\oplus g}$
is a convex subset of $C\oplus D$. Moreover, $f_n\oplus g_n \to
f\oplus g$ in $\hlzer$ if  and only if $f_n\to f$ and $g_n\to g$ in
$\lzer$. Therefore, we can use upper semi-continuity of the maps 
$\Phi(\cdot,g)$ and $-\Phi(f,\cdot)$, valid for any $f\in C$ and $g\in
D$,  to conclude that 
 $G_{f\oplus g}$ is closed in $\hat{\PP}$ for each $f\oplus g\in C\oplus
D$. Finally, since $\hat{\PP}[ f\oplus g\geq M]= \tot( \PP[f\geq
M]+\PP[g\geq M])$, it is clear that $C\oplus D$ is bounded in
probability. Therefore, $G_{f\oplus g}$ is a family of closed and convex
subsets of a convex-compact set $C\oplus D$. Next, we
state and prove an auxiliary claim.
\begin{claim}
\label{cla:KKM}
The family $\sets{G_{f\oplus g}}{f\oplus g\in C\oplus D}$ has the KKM property.
\end{claim} 
To prove Claim \ref{cla:KKM}, we assume, to the contrary, that there
exist $\tilde{f}\oplus\tilde{g}\in C\oplus D$, a finite family
$f_1\oplus g_1,\dots, f_m\oplus g_m$ in $C\oplus D$ and a set of
non-negative weights $\{\alpha_k\}_{k\in\ft{1}{m}}$ with $\sum_{k=1}^m
\alpha_k=1$, such that $\tilde{f}\oplus\tilde{g}=\sum_{k=1}^m \alpha_k\,
 f_k\oplus g_k $, but $\tilde{f}\oplus\tilde{g}\not\in G_{f_k\oplus
  g_k}$ for $k=\ft{1}{m}$, i.e., $\Phi(f_k,\tilde{g})>\Phi(\tilde{f},
g_k)$, for $k=\ft{1}{m}$.  Then, we have the following string of
inequalities:
\[    \Phi(\tilde{f},\tilde{g})\geq  
\sum_{k=1}^m \alpha_k \Phi(f_k, \tilde{g}) > \sum_{k=1}^n
\alpha_k \Phi(\tilde{f}, g_k)\geq \Phi(\tilde{f},\tilde{g}). \]
Evidently, this is 
a contradiction. Therefore, Claim \ref{cla:KKM} is established.

We continue the proof invoking  of  Theorem 
\ref{thm:KKM}; its assumptions are satisfied, thanks to
 Claim \ref{cla:KKM} and the discussion
preceding it. By Theorem \ref{thm:KKM}, there exists $f_0\oplus g_0\in
C\oplus D$ such that $(f_0,g_0)\in G_{f\oplus g}$ for all $f\in C$ and
all $g\in D$, i.e., 
$ \Phi(f,g_0)\leq \Phi(f_0,g)$ for all $f\in C$ and $g\in
D$. Substituting $f=f_0$ or $g=g_0$, we obtain
\[ \Phi(f,g_0)\leq \Phi(f_0,g_0)\leq \Phi(f_0,g),\text{ for all } f\in
C,\, g\in D.\]
The  last statement of the theorem -  equation \eqref{equ:supinf} -   follows
readily.
\end{proof}

\subsection{An excess-demand theorem  under convex compactness}
Our last application is an extension of the Walrasian excess-demand theorem. 
Similarly to the situations described
above, versions of excess-demand theorem have been proved in various
settings (see, e.g., \cite{ArrDeb54}, \cite{Mck59} and Exercise C.7,
p. 179 in \cite{GraDug03}), but, to the best of our knowledge, always
under the assumption of local convexity and compactness.

Typically, the excess-demand theorem is applied to a function $F$ of
the type $F(x,y)=\scl{\Delta(x)}{y}$, where $x$ is thought of as a
price-system, $\Delta(x)$ is the excess aggregate demand (or a
zero-preserving transformation thereof) for the bundle of all
commodities, and $y$ is a test function. The conclusion $F(x_0,y)\leq
0$, for every $y$, is then used to establish the equality
$\Delta(x_0)=0$ - a stability (equilibrium) condition for the market
under consideration.
\begin{theorem}
  \label{thm:excess-dem}
  Let $C$ be a convexly compact subset of a topological vector
  space $X$, and let $D\subseteq C$ be convex and closed.  
Let the mapping $F:C\times D\to \R$
  satisfy the following properties
  \begin{enumerate}
  \item \label{ite:losed} for each $y\in D$, the set $\sets{x\in
      C}{ F(x,y)\leq 0}$ is closed and convex,
  \item \label{ite:conv-in-y} for each $x\in C$, the function $y\mapsto F(x,y)$ is concave,
    and 
  \item \label{ite:Walras} for each $y\in D$, $F(y,y)\leq 0$.
  \end{enumerate}
  Then there exists $x_0\in C$ such that $F(x_0,y)\leq 0$, for all
  $y\in D$.
\end{theorem}
\begin{proof} 
  Define the family $\set{F_y}_{y\in D}$ of closed and convex subsets
  of $C$ by $F_y=\{ x\in C\,:\, F(x,y)\leq 0\}$. In order to show that
  it has the KKM property, we assume, to the contrary, that there
  exist $y_1,\dots, y_m$ in $D$ and a set of non-negative weights
  $\alpha_1,\dots, \alpha_m$ with $\sum_{k=1}^m \alpha_k=1$ such that
  $\tilde{y}:=\sum_{k=1}^m \alpha_k y_k \not\in \cup_{k=1}^m F_{y_k}$,
  i.e., $F(\tilde{y},y_k)>0$, for $k=\ft{1}{m}$.  Then, by
  property {\em \ref{ite:conv-in-y}.} of $F$, 
we have $0 \geq F(\tilde{y},\tilde{y})
  \geq \sum_{k=1}^m \alpha_k F(\tilde{y}, y_k) > 0$ -- a
  contradiction. Therefore, by Theorem \ref{thm:KKM}, there exists
  $x_0\in C$ such that $x_0\in F_y$ for all $y\in D$, i.e.,
  $F(x_0,y)\leq 0$, for all $y\in D$.
\end{proof}
It is outside the scope of the present paper to construct in detail
concrete equilibrium situations where the excess-demand function above
is used to guarantee the existence of a Walrasian equilibrium - this will
be a topic of our future work.

\ifx \cprime \undefined \def \cprime {$\mathsurround=0pt '$}\fi\ifx \k
  \undefined \let \k = \c \fi\ifx \scr \undefined \let \scr = \cal \fi\ifx
  \soft \undefined \def \soft {\relax}\fi

% \bibliographystyle{plainnat}
% \bibliography{mybib}

\end{document}